\documentclass{amsart}[12pt]

\usepackage{amsmath}
\usepackage{amsfonts}
\usepackage{amssymb}
\DeclareMathOperator*{\dist}{dist}

\newtheorem{theorem}{Theorem}
\newtheorem{lemma}{Lemma}

\newtheorem{corollary}{Corollary}

\theoremstyle{definition}
\newtheorem{definition}{Definition}

\newtheorem{remark}{Remark}

\begin{document}
\title[ Integral operators on analytic Morrey spaces]%
{ Integral operators on analytic Morrey spaces}

\author{Pengtao Li, junming liu \& Zengjian Lou*}

\address{Department of Mathematics, Shantou University, Shantou, Guangdong,
515063, \newline P.~R.~China} \email{ptli@stu.edu.cn}

 \email{08jmliu@stu.edu.cn}

 \email{zjlou@stu.edu.cn}

\begin{abstract}
In this note, we study the boundedness of integral operators $I_{g}$
and $T_{g}$ on analytic Morrey spaces. Furthermore, the norm and
essential norm of those operators are given.
\end{abstract}
\thanks{This work was supported by NNSF of China (Grant No. 11171203 and No. 11201280) and NSF of Guangdong Province
(Grant no. 10151503101000025 and No. S2011010004511).}
 \thanks{*Corresponding author.}

\maketitle

\section{Introduction}
 Morrey spaces were
initially introduced by Morrey \cite{MCB} in 1938.  As  useful
tools, Morrey spaces play an important role in the study of harmonic
analysis and partial differential equations. See  Taylor \cite{T},
Olsen \cite{Ol}, Kukavica \cite{IK}, Palagachev and Softova
\cite{PS}, and the reference therein.

In recent decades, in real and complex settings, Morrey spaces have
been studied extensively. For example, in Euclidean spaces
$\mathbb{R}^{n}$, Adams and Xiao studied Morrey spaces by potential
theory and Hausdorff capacity in \cite{AX} and \cite{AX1}.  Duong,
Yan and Xiao \cite{DXY} characterized Morrey spaces by the operators
with heat kernel bounds. The multipliers of Morrey spaces were
studied by Gilles and Rieusset \cite{GR}. In the unit disc
$\mathbb{D}$, the analytic Morrey spaces, $\mathcal{L}^{2,\lambda}$,
were introduced and studied by Wu and Xie in \cite{WX}. Xiao and Xu
\cite{XX} studied the composition operators of
$\mathcal{L}^{2,\lambda}$ spaces. Cascante, F\`{a}brega and Ortega
\cite{CFO} studied the Corona theorem of $\mathcal{L}^{2,\lambda}$.

For analytic functions $g$ in $\mathbb{D}$, the Volterra type
operator $T_{g}$ is defined as
$$T_{g}f(z)=\int_{0}^{z}f(w)g'(w)dw, $$
on the space of analytic functions $f$ in $\mathbb{D}$. The
operator $T_{g}$ was firstly investigated on Hardy spaces by
Pommerenke \cite{PC}. Another integral operator related to $T_{g}$
(denoted by $I_{g}$) is defined by
$$I_{g}f(z)=\int_{0}^{z}f'(w)g(w)dw.$$
The boundedness and compactness of  $I_{g}$ and $T_{g}$ between
spaces of analytic functions were studied by many authors. The
$T_{g}$ on Hardy spaces and Bergman spaces were studied by Aleman
and Cima in \cite{AC},  Aleman and Siskakis in \cite{AS, AS1}.
Siskakis and Zhao \cite{SZ} studied $T_{g}$ on the space $BMOA$.
Xiao \cite{XJ1} considered $I_{g}$ and $T_{g}$ on $Q_{p}$ spaces.
Constantin \cite{CO} studied the boundedness and the compactness of
$T_{g}$ on Fock spaces. In \cite{WZ}, Wu  considered $T_{g}$ from
Hardy spaces to analytic Morrey spaces. The closed ranges of $I_{g}$
and $T_{g}$ are studied by Anderson \cite{AA}.  See Aleman
\cite{ALA} and its references
  for more information on those integral operators.

The essential norm can be seen as a useful tool to study the
operators on function spaces. By the essential norms,
 we can understand better the relation between bounded operators and compact operators.
The essential norms were first introduced by Shapiro \cite{Sh} to
study the composition operators. Recently, the essential norms were
 applied to study the integral operators $I_{g}$ and $T_{g}$, see
Laitila, Miihkinen and Nieminen \cite{JSP}, and Liu, Lou and Xiong
\cite{LLX}. We refer the reader to \cite{Sh},  \cite{JSP} and
\cite{LLX} for further information.

In this paper, we consider the operators
$$I_{g},
T_{g}:\mathcal{L}^{2,\lambda}(\mathbb{D})\rightarrow
\mathcal{L}^{2,\lambda}(\mathbb{D}),\ 0<\lambda<1$$
 and
 $$I_{g},
T_{g}:\ H^{p}(\mathbb{D}) \rightarrow \mathcal{L}^{2,
1-\frac{2}{p}}(\mathbb{D}),\ 2<p\leq \infty.$$ The aim is to study
the boundedness of $I_{g}$ and $T_{g}$, and to estimate the norms
and essential norms of $I_{g}$ and $T_{g}$. We will prove that the
norms of $I_{g}$ and $T_{g}$ is equivalent to $\|g\|_{\infty}$ and
$\|g\|_{BMOA}$, respectively. By these equivalence, we could obtain
the necessary and sufficient conditions of the boundedness of
$I_{g}$ and $T_{g}$. See Theorems 1-4. On essential norms, we prove
that the essential norm of $I_{g}$ is equivalent to
$\|g\|_{\infty}$, and that the essential norm of $T_{g}$ is
equivalent to the distance between $g$ and the space $VMOA$, see
Theorems 5-8. As corollaries, we obtain the sufficient and necessary
conditions of the compactness for $I_{g}$ and $T_{g}$, respectively.
We should point out that these necessary and sufficient conditions
can be also proved without using the essential norms. Compared with
such method, our results provide more information.

 The rest of this paper is organized as follows. In
 Section 2, we state some notations and preliminaries which will be used in
 the sequel. Section 3 is devoted to the study of the boundedness of
 $I_{g}$ and $T_{g}$ and of the norm estimate of those integral operators.
  The essential norms
 of $I_{g}$ and $T_{g}$ are given in Section 4.

{\it Notations}:\ For two functions $F$ and $G$, if there is a
constant $C>0$ dependent only on indexes $p, \lambda, \cdots$ such
that $F\leq C G$, then we say that $F\lesssim G$. Furthermore,
denote that $F\approx G$ (F is comparable with G) whenever
$F\lesssim G\lesssim F$.

\section{Notations and Preliminaries}
 Let $\mathbb{D}$ and  $\partial \mathbb{D}=\{z:|z|=1\}$
 denote respectively the open unit disc and the unit circle in the
complex plane $\mathbb{C}$. Let $H(\mathbb{D})$ be the space of all
analytic functions on $\mathbb{D}$ and $dA(z)=\frac{1}{\pi}dxdy$ the
normalized area Lebesgue measure.

For $0<p<\infty$, the Hardy space $H^{p}(\mathbb{D})$ consists of
all functions $f\in H(\mathbb{D})$ with
$$\|f\|_{H^{p}}^{p}=\sup_{0<r<1}\frac{1}{2\pi}\int_{0}^{2\pi}|f(re^{i\theta})|^{p}d\theta<\infty.$$

 For an arc $I\subset \partial
\mathbb{D}$, let $|I|=\frac{1}{2\pi}\int_{I}|d\zeta|$ be the
normalized length of $I$,
$$f_{I}=\frac{1}{|I|}\int_{I}f(\zeta)\frac{|d\zeta|}{2\pi},\ \ f\in H(\mathbb{D}),$$
and $S(I)$ be the Carleson box based on $I$ with
$$S(I)=\{z\in \mathbb{D}:1-|I|\leq |z|<1~\mbox{and}~\frac{z}{|z|}\in I\}.$$
Clearly, if $I=\partial \mathbb{D}$, then $S(I)=\mathbb{D}$.

\begin{definition}
For $0<p<\infty$, we say that a non-negative measure $\mu$ on
$\mathbb{D}$ is a $p$-Carleson measure if
$$\sup_{I\subset \partial \mathbb{D}}\frac{\mu(S(I))}{|I|^{p}}<\infty.$$
If $p=1$, $p$-Carleson measure is the classical Carleson measure.
\end{definition}

The following theorem due to L. Carleson is a significant result on
Carleson measure (see, for example, \cite{GJB} and \cite{zhu}).

\noindent {\bf Theorem A}. Suppose that $\mu$ is a non-negative
measure on $\mathbb{D}$. Then $\mu$ is a Carleson measure if and
only if the following inequality
$$\int_{\mathbb{D}}|f(z)|^{2}d\mu(z)\lesssim \|f\|_{H^{2}}^{2}=\int_{0}^{2\pi}|f(e^{i\theta})|^{2}d\theta$$
holds for all $f$ in Hardy space $ H^{2}(\mathbb{D})$. Moreover,
\begin{equation}\nonumber
\sup_{\|f\|_{H^{2}}=1} \int_{\mathbb{D}}|f(z)|^{2}d\mu(z)\approx
\sup_{I\subset
\partial \mathbb{D}}\frac{\mu(S(I))}{|I|}.
\end{equation}

Recall that $BMOA$ is the set of $f$ in the Hardy space $
H^{2}(\mathbb{D})$ whose boundary value functions satisfies
$$\|f\|_{*}=\sup_{I\subset \partial \mathbb{D}}\Big( \frac{1}{|I|}\int_{I}|f(e^{i\theta})-f_{I}|^{2}d\theta\Big)^{1/2}<\infty.$$
 The norm of functions $f$ in $BMOA$ can be expressed as $\|f\|=|f(0)|+\|f\|_{*}$.
From \cite{GD}, we know that $\|f\|$ is comparable with the norm
$$\|f\|_{BMOA}=|f(0)|+\sup_{a\in \mathbb{D}}\|f\circ\sigma_{a}-f(a)\|_{H^{2}},$$
where $a\in \mathbb{D}$ and $\sigma_{a}$ is the M\"{o}bius
transformation with
$$\sigma_{a}(z)=\frac{a-z}{1-\overline{a}z},\ \ z\in \mathbb{D}.$$
Furthermore, if
$$\lim_{|I|\rightarrow 0} \frac{1}{|I|}\int_{I}|f(e^{i\theta})-f_{I}|^{2}d\theta=0, $$
we say that $f\in VMOA$, the space of analytic functions with
vanishing mean oscillation, we also know that $f\in VMOA$ if and
only if
$$\lim_{|a|\rightarrow 1}\|f\circ\sigma_{a}-f(a)\|_{H^{2}}=0.$$
For more information on $BMOA$ and $VMOA$, see \cite{GD}.

The following theorem (Theorem 6.5 of \cite{GD} or Theorem 4.1.1 of
\cite{XJ1}) is a Carleson measure characterization of $BMOA$.

\noindent {\bf Theorem B}. Let $f\in H(\mathbb{D})$, then $f\in
BMOA$ if and only if the measure $\mu_{f}$ with
$d\mu_{f}=|f'(z)|^{2}(1-|z|^{2})dA(z)$ is a Carleson measure.
Moreover,
\begin{equation}\nonumber
\|f\|_{BMOA}\approx |f(0)|+\Big(\sup_{I\subset
\partial \mathbb{D}}\frac{\mu_{f}(S(I))}{|I|} \Big)^{1/2}.
\end{equation}

\begin{definition}
Let $0< \lambda\leq 1$. The Morrey space $\mathcal{L}^{2,
\lambda}(\mathbb{D})$ is the set of all $f$ belongs to the Hardy
space $ H^{2}(\mathbb{D})$ such that
$$\sup_{I\subset \partial \mathbb{D}}\Big(\frac{1}{|I|^{\lambda}}
\int_{I}|f(\zeta)-f_{I}|^{2}\frac{|d\zeta|}{2\pi}\Big)^{1/2}<\infty.$$
\end{definition}

Clearly, $\mathcal{L}^{2, 1}(\mathbb{D})=BMOA$. The following lemma
gives some equivalent conditions of
$\mathcal{L}^{2,\lambda}(\mathbb{D})$ (see Theorem 3.21 of
\cite{XJ2} or Theorem 3.1 of \cite{Wulan}).

\begin{lemma}\label{le1}
Suppose that $0<\lambda< 1$ and $f\in H(\mathbb{D})$. Let
$\sigma_{a}(z)=\frac{a-z}{1-\overline{a}z}$. Then the following
statements are equivalent.
\begin{enumerate}
\item[(i)] $f\in \mathcal{L}^{2, \lambda}(\mathbb{D})$;
\item[(ii)] $\sup\limits_{I\subset \partial
\mathbb{D}}\frac{1}{|I|^{\lambda}}\int_{S(I)}|f'(z)|^{2}(1-|z|^{2})dA(z)<\infty$;
\item[(iii)] $\sup\limits_{a\in
\mathbb{D}}(1-|a|^{2})^{1-\lambda}\int_{\mathbb{D}}|f'(z)|^{2}(1-|\sigma_{a}(z)|^{2})dA(z)<\infty$;
\item[(iv)] $\sup\limits_{a\in
\mathbb{D}}(1-|a|^{2})^{1-\lambda}\int_{\mathbb{D}}|f'(z)|^{2}\log
\frac{1}{|\sigma_{a}(z)|}dA(z)<\infty$.
\end{enumerate}
\end{lemma}

The norm of functions $f\in \mathcal{L}^{2, \lambda}(\mathbb{D})$
can be defined as follows
$$\|f\|_{\mathcal{L}^{2, \lambda}}=|f(0)|+\sup_{I \subset \partial\mathbb{D}}
\Big(\frac{1}{|I|^{\lambda}}\int_{S(I)}|f'(z)|^{2}(1-|z|^{2})dA(z)\Big)^{1/2}.$$

\begin{remark}
From the lemma above, it is easy to see that for $f\in
\mathcal{L}^{2, \lambda}(\mathbb{D})$,
\begin{equation}\nonumber
\begin{split}\|f\|_{\mathcal{L}^{2, \lambda}}&\approx
|f(0)|+\sup\limits_{a\in
\mathbb{D}}\Big((1-|a|^{2})^{1-\lambda}\int_{\mathbb{D}}|f'(z)|^{2}(1-|\sigma_{a}(z)|^{2})dA(z)\Big)^{1/2}
\\&\approx |f(0)|+\sup\limits_{a\in
\mathbb{D}}\Big((1-|a|^{2})^{1-\lambda}\int_{\mathbb{D}}|f'(z)|^{2}\log
\frac{1}{|\sigma_{a}(z)|}dA(z) \Big)^{1/2}.
\end{split}
\end{equation}
\end{remark}

Now we give a result about the growth rate of functions in
$\mathcal{L}^{2, \lambda}(\mathbb{D})$.
\begin{lemma}\label{le2}
Let $0<\lambda<1$. If $f\in \mathcal{L}^{2, \lambda}(\mathbb{D})$,
then
$$|f(z)|\lesssim \frac{\|f\|_{\mathcal{L}^{2, \lambda}}}{(1-|z|^{2})^{\frac{1-\lambda}{2}}},\ \ z\in \mathbb{D}.$$
\end{lemma}
\begin{proof}
For any $b\in \mathbb{D}$, we have
\begin{equation}\nonumber
\begin{split}
\|f\|_{\mathcal{L}^{2, \lambda}}&\approx|f(0)|+\sup\limits_{a\in
\mathbb{D}}\Big((1-|a|^{2})^{1-\lambda}\int_{\mathbb{D}}|f'(z)|^{2}(1-|\sigma_{a}(z)|^{2})dA(z)\Big)^{1/2}\\&\geq
\Big((1-|b|^{2})^{1-\lambda}\int_{\mathbb{D}}|f'(z)|^{2}(1-|\sigma_{b}(z)|^{2})dA(z)\Big)^{1/2}
\\&=\Big((1-|b|^{2})^{1-\lambda}\int_{\mathbb{D}}|f'(\sigma_{b}(w))|^{2}|\sigma_{b}'(w)|^{2}(1-|w|^{2})dA(w)\Big)^{1/2}
\\&\gtrsim (1-|b|^{2})^{\frac{3-\lambda}{2}}|f'(b)|,
\end{split}
\end{equation}
where we used Lemma 4.12 of \cite{zhu} in the last inequality. Thus
\begin{equation}\nonumber
\begin{split}
|f(b)-f(0)|&=\Big|b\int_{0}^{1}f'(bt)dt \Big|\leq
|b|\int_{0}^{1}|f'(bt)|dt
\\&\lesssim\|f\|_{\mathcal{L}^{2,
\lambda}}|b|\int_{0}^{1}\frac{1}{(1-|b|t)^{\frac{3-\lambda}{2}}}dt
\\&\lesssim\frac{\|f\|_{\mathcal{L}^{2,
\lambda}}}{(1-|b|)^{\frac{1-\lambda}{2}}}.
\end{split}
\end{equation}
Since the point $b$ is arbitrary, we get
$$|f(z)|\lesssim \frac{\|f\|_{\mathcal{L}^{2, \lambda}}}{(1-|z|^{2})^{\frac{1-\lambda}{2}}},\ \ z\in \mathbb{D}.$$
\end{proof}

\section{\bf Boundedness and Norm estimate of $I_{g}$ and $T_{g}$}
In this section,  we prove the boundedness and estimate the norms of
$I_{g}$ and $T_{g}$.
  The following two lemmas will be used through this paper.

\begin{lemma}\label{le3}
(\cite[Lemma 1]{ZR}) Suppose that $s>-1$ and $r, t>0$. If $t<s+2<r$,
then we have
$$\int_{\mathbb{D}}\frac{(1-|z|^{2})^{s}}{|1-\overline{b}z|^{r}|1-\overline{a}z|^{t}}dA(z)\lesssim
\frac{1}{(1-|b|^{2})^{r-s-2}|1-\overline{a}b|^{t}},$$ where $a, b
\in \mathbb{D}$.
\end{lemma}

\begin{lemma}\label{le4}
Let $0<\lambda< 1$ and $b\in \mathbb{D}$. Then functions
$f_{b}(z)=(1-|b|^{2})^{\frac{\lambda-1}{2}}(\sigma_{b}(z)-b)$ and
$F_{b}(z)=(1-|b|^{2})(1-\overline{b}z)^{\frac{\lambda-3}{2}}$ belong
to $\mathcal{L}^{2, \lambda}(\mathbb{D})$. Moreover, we have
$\|f_{b}(z)\|_{\mathcal{L}^{2, \lambda}}\lesssim 1$ and
$\|F_{b}(z)\|_{\mathcal{L}^{2, \lambda}}\lesssim 1$, where the
constants implicit are independent of $z$ and $b$.
\end{lemma}
\begin{proof}
For $b\in \mathbb{D}$, by Lemma $\ref{le3}$, we have
\begin{equation}\nonumber
\begin{split}
\|f_{b}\|_{\mathcal{L}^{2, \lambda}}^{2}&\approx\sup_{a\in
\mathbb{D}}
(1-|a|^{2})^{1-\lambda}\int_{\mathbb{D}}|f_{b}'(z)|^{2}(1-|\sigma_{a}(z)|^{2})dA(z)\\&
=\sup_{a\in \mathbb{D}}
(1-|a|^{2})^{2-\lambda}(1-|b|^{2})^{\lambda+1}\int_{\mathbb{D}}
\frac{(1-|z|^{2})}{|1-\overline{b}z|^{4}|1-\overline{a}z|^{2}}dA(z)
\\&\lesssim\sup_{a\in
\mathbb{D}}\frac{(1-|a|^{2})^{2-\lambda}(1-|b|^{2})^{\lambda}}{|1-\overline{a}b|^{2}}\lesssim
1
\end{split}
\end{equation}
and

\begin{equation}\nonumber
\begin{split}\|F_{b}\|_{\mathcal{L}^{2, \lambda}}^{2}&\lesssim 1+\sup_{a\in\mathbb{ D}}(1-|a|^{2})^{2-\lambda}(1-|b|^{2})^{2}\int_{\mathbb{D}}
\frac{(1-|z|^{2})}{|1-\overline{b}z|^{5-\lambda}|1-\overline{a}z|^{2}}dA(z)
\\&\lesssim 1+ \sup_{a\in
\mathbb{D}}\frac{(1-|a|^{2})^{2-\lambda}(1-|b|^{2})^{\lambda}}{|1-\overline{a}b|^{2}}\lesssim
1.
\end{split}
\end{equation}
\end{proof}

We first consider the boundedness of $I_{g}$ on $\mathcal{L}^{2,
\lambda}(\mathbb{D})$.

\begin{theorem}\label{th1}
Let $0<\lambda<1$ and $g\in H(\mathbb{D})$. Then $I_{g}$ is bounded
on $\mathcal{L}^{2, \lambda}(\mathbb{D})$ if and only if $g$ belongs
to $H^{\infty}(\mathbb{D})$, the space of bounded analytic functions
on $\mathbb{D}$. Moreover the operator norm satisfies
$$\|I_{g}\|\approx \|g\|_{\infty},$$ where
$\|g\|_{\infty}=\sup_{z\in \mathbb{D}}|g(z)|$.
\end{theorem}
\begin{proof}
Let $g\in H^{\infty}(\mathbb{D})$. For any $f\in \mathcal{L}^{2,
\lambda}(\mathbb{D})$, we have
\begin{equation}\nonumber
\begin{split}
\|I_{g}f\|_{\mathcal{L}^{2, \lambda}}&=\sup_{I \subset
\partial \mathbb{D}}
\Big(\frac{1}{|I|^{\lambda}}\int_{S(I)}|f'(z)|^{2}|g(z)|^{2}(1-|z|^{2})dA(z)\Big)^{1/2}
\\&\leq \|f\|_{\mathcal{L}^{2,
\lambda}}\sup_{z\in \mathbb{D}}|g(z)|.
\end{split}
\end{equation}
This leads the boundedness of $I_{g}$ and $\|I_{g}\|\leq
\|g\|_{\infty}$.

On the other hand, if $I_{g}$ is bounded on $\mathcal{L}^{2,
\lambda}(\mathbb{D})$. For any $b\in \mathbb{D}$, suppose that
$f_{b}$ is defined as in Lemma $\ref{le4}$, then
\begin{equation}\nonumber
\begin{split}
\|I_{g}\|&\gtrsim \|I_{g}f_{b}\|_{\mathcal{L}^{2,
\lambda}}\\&\approx \sup_{a\in
\mathbb{D}}\Big((1-|a|^{2})^{1-\lambda}\int_{\mathbb{D}}|f_{b}'(z)|^{2}|g(z)|^{2}(1-|\sigma_{a}(z)|^{2})dA(z)
\Big)^{1/2}\\&\geq\Big(\int_{\mathbb{D}}
|\sigma_{b}'(z)|^{2}|g(z)|^{2}(1-|\sigma_{b}(z)|^{2})dA(z)\Big)^{1/2}
\\&=\Big(\int_{\mathbb{D}}
|g(\sigma_{b}(w))|^{2}(1-|w|^{2})dA(w)\Big)^{1/2}
\\&\gtrsim |g(b)|,
\end{split}
\end{equation}
where we used Lemma 4.12 of \cite{zhu} again in the last inequality.
Since  $b$ is arbitrary, we have $\|I_{g}\|\gtrsim \|g\|_{\infty}$.
Theorem $\ref{th1}$ is proved.
\end{proof}

As a main result of this section, we give the boundedness of $T_{g}$
on $\mathcal{L}^{2, \lambda}(\mathbb{D})$ in the next theorem.

\begin{theorem}\label{th2}
Suppose that $0<\lambda<1$ and $g\in H(\mathbb{D})$. Then $T_{g}$ is
bounded on $\mathcal{L}^{2, \lambda}(\mathbb{D})$ if and only if
$g\in BMOA$. Moreover, $\|T_{g}\|\approx \|g\|_{BMOA}$.
\end{theorem}

\begin{proof}
Suppose  $g\in BMOA$. For $f\in  \mathcal{L}^{2,
\lambda}(\mathbb{D})$ and  any arc $I\subset
\partial{D}$, let $\zeta\in \partial \mathbb{D}$ be the center of arc $I$ and
  $b=(1-|I|)\zeta\in \mathbb{D}$. We have
\begin{equation}\nonumber
\begin{split}
\frac{1}{|I|^{\lambda}}&\int_{S(I)}|(T_{g}f)'(z)|^{2}(1-|z|^{2})dA(z)\\&=
\frac{1}{|I|^{\lambda}}\int_{S(I)}|f(z)|^{2}|g'(z)|^{2}(1-|z|^{2})dA(z)\\&\leq
\frac{2}{|I|^{\lambda}}\int_{S(I)}|f(b)|^{2}|g'(z)|^{2}(1-|z|^{2})dA(z)\\&\qquad+
\frac{2}{|I|^{\lambda}}\int_{S(I)}|f(z)-f(b)|^{2}|g'(z)|^{2}(1-|z|^{2})dA(z)
\\&\triangleq I_{1}+I_{2}.
\end{split}
\end{equation}

We first estimate the term $I_{1}$. By Lemma $\ref{le2}$, we get
$$|f(b)|\lesssim  \frac{\|f\|_{\mathcal{L}^{2, \lambda}}}{(1-|b|^{2})^{\frac{1-\lambda}{2}}}
\lesssim\frac{\|f\|_{\mathcal{L}^{2,
\lambda}}}{|I|^{\frac{1-\lambda}{2}}}.$$ It follows from Theorem B
that
$$I_{1}=\frac{2}{|I|^{\lambda}}\int_{S(I)}|f(b)|^{2}|g'(z)|^{2}(1-|z|^{2})dA(z)\lesssim \|g\|_{BMOA}^{2}\|f\|_{\mathcal{L}^{2,
\lambda}}^{2}.$$

 Now we estimate the term $I_{2}$. Since $1-|\sigma_{b}(z)|^{2}\gtrsim
 \frac{1-|z|^{2}}{|I|}$ for all $z\in S(I)$, then
 \begin{equation}\nonumber
 \begin{split}
 I_{2}&=\frac{2}{|I|^{\lambda}}\int_{S(I)}|f(z)-f(b)|^{2}|g'(z)|^{2}(1-|z|^{2})dA(z)\\&\lesssim
 |I|^{1-\lambda}\int_{S(I)}|f(z)-f(b)|^{2}|g'(z)|^{2}(1-|\sigma_{b}(z)|^{2})dA(z)
\\&\lesssim |I|^{1-\lambda}\int_{\mathbb{D}}|f\circ\sigma_{b}(w)-f(b)|^{2}|(g\circ\sigma_{b})'(w)|^{2}(1-|w|^{2})dA(w)
\\&\lesssim
(1-|b|^{2})^{1-\lambda}\int_{\mathbb{D}}|f\circ\sigma_{b}(w)-f(b)|^{2}|(g\circ\sigma_{b})'(w)|^{2}(1-|w|^{2})dA(w).
 \end{split}
 \end{equation}
Since $g\in BMOA$, then $g\circ\sigma_{b}\in BMOA$ and
$|(g\circ\sigma_{b})'(w)|^{2}(1-|w|^{2})dA(w)$ is a Carleson measure
by Theorem B. Note that $f\in \mathcal{L}^{2, \lambda}(\mathbb{D})
\subset H^{2}(\mathbb{D})$, then $f\circ\sigma_{b}-f(b)\in
H^{2}(\mathbb{D})$. Combining this with Theorem A yields
\begin{equation}\nonumber
\begin{split}
 I_{2}&\lesssim
 (1-|b|^{2})^{1-\lambda}\|g\circ\sigma_{b}\|_{BMOA}^{2}\int_{0}^{2\pi}|f\circ\sigma_{b}(e^{i\theta})-f(b)|^{2}d\theta
 \\&\lesssim
 (1-|b|^{2})^{1-\lambda}\|g\|_{BMOA}^{2}\int_{\mathbb{D}}|f'(z)|^{2}\log\frac{1}{|\sigma_{b}(z)|}dA(z)
 \\&\lesssim\|g\|_{BMOA}^{2}\|f\|_{\mathcal{L}^{2, \lambda}}^{2},
\end{split}
\end{equation}
where we used Littlewood-Paley identity in the second inequality
(see, for example, \cite{GJB} or \cite{zhu}). Thus,
$$\|T_{g}f\|_{\mathcal{L}^{2, \lambda}}^{2}\lesssim I_{1}+I_{2}\lesssim\|g\|_{BMOA}^{2}\|f\|_{\mathcal{L}^{2, \lambda}}^{2}.$$
That is $\|T_{g}\|\lesssim \|g\|_{BMOA}$.

On the other hand, suppose that $T_{g}$ is bounded on
$\mathcal{L}^{2, \lambda}(\mathbb{D})$. For any $I\subset \partial
\mathbb{D}$, let $b=(1-|I|)\zeta\in \mathbb{D}$, where $\zeta$ is
the center of $I$. Then
$$(1-|b|^{2})\approx|1-\overline{b}z|\approx |I|,\ \ z\in S(I).$$
Using the test function $F_{b}$ as in Lemma 4, we get
\begin{equation}\nonumber
\begin{split}
\frac{1}{|I|}&\int_{S(I)}|g'(z)|^{2}(1-|z|^{2})dA(z)
\\&\lesssim
\frac{1}{|I|^{\lambda}}\int_{S(I)}|F_{b}(z)|^{2}|g'(z)|^{2}(1-|z|^{2})dA(z)
\\&=\frac{1}{|I|^{\lambda}}\int_{S(I)}|(T_{g}F_{b})'(z)|^{2}(1-|z|^{2})dA(z)
\\&\leq \|T_{g}F_{b}\|_{\mathcal{L}^{2, \lambda}}^{2}
\\&\leq \|T_{g}\|^{2}\|F_{b}\|_{\mathcal{L}^{2, \lambda}}^{2}
\\&\lesssim \|T_{g}\|^{2}.
\end{split}
\end{equation}
Since $I$ is arbitrary, we obtain $g\in BMOA$ and
$\|g\|_{BMOA}\lesssim \|T_{g}\|$.
 The proof is completed.
\end{proof}
\begin{remark}
Theorem $\ref{th1}$ holds for $\lambda=1$ (see \cite{XJ1}). But
Theorem $\ref{th2}$ is not ture for $\lambda=1$. In fact, Siskakis
and Zhao \cite{SZ} proved that $T_{g}$ is bounded on $BMOA$ if and
only if $g$ belongs to logarithmic $BMOA$ space.
\end{remark}

For $g\in H(\mathbb{D})$, the multiplication operator  $M_{g}$ is
defined by $M_{g}f(z)=f(z)g(z)$. It is easy to see that $M_{g}$ is
related with $I_{g}$ and $T_{g}$ by
$$M_{g}f(z)=f(0)g(0)+I_{g}f(z)+T_{g}f(z).$$

\begin{corollary}
Let $0<\lambda<1$ and $g\in H(\mathbb{D})$. Then the multiplication
operator $M_{g}$ is bounded on $\mathcal{L}^{2,
\lambda}(\mathbb{D})$ if and only if $g\in H^{\infty}(\mathbb{D})$.
\end{corollary}
\begin{proof}
If $g\in H^{\infty}$, then $I_{g}$ and $ T_{g}$ are both bounded on
$\mathcal{L}^{2, \lambda}(\mathbb{D})$ by Theorems $\ref{th1}$ and
$\ref{th2}$. So $M_{g}$ is bounded on $\mathcal{L}^{2,
\lambda}(\mathbb{D})$.

If $M_{g}$ is bounded on $L^{2,\lambda}(\mathbb{D})$, consider the
function $F_{b}(z)$ in Lemma $\ref{le4}$. Applying Lemma
$\ref{le2}$, we have
\begin{equation}\nonumber
\begin{split}
\Big|\frac{1-|b|^{2}}{(1-\overline{b}z)^{\frac{3-\lambda}{2}}}g(z)\Big|&\lesssim
\|M_{g}F_{b}\|_{\mathcal{L}^{2,
\lambda}}\frac{1}{(1-|z|^{2})^{\frac{1-\lambda}{2}}}\\&\lesssim\|M_{g}\|\frac{1}{(1-|z|^{2})^{\frac{1-\lambda}{2}}}.
\end{split}
\end{equation}
Taking $z=b$ in the inequality above yields $|g(b)|\lesssim
\|M_{g}\|$ and then $g\in H^{\infty}(\mathbb{D})$ by the
arbitrariness of $b$.
\end{proof}

 We now characterize the boundedness of $I_{g}:
H^{p}(\mathbb{D})\rightarrow
\mathcal{L}^{2,1-\frac{2}{p}}(\mathbb{D}) (2<p\leq \infty)$.

\begin{theorem}\label{th31}
Let $2<p\leq \infty$ and $g\in H(\mathbb{D})$. Then $I_{g}:
H^{p}(\mathbb{D})\rightarrow
\mathcal{L}^{2,1-\frac{2}{p}}(\mathbb{D})$ is bounded if and only if
$g\in H^{\infty}(\mathbb{D})$. Moreover $\|I_{g}\|\approx
\|g\|_{\infty}$.
\end{theorem}

\begin{proof}
Suppose $I_{g}: H^{p}(\mathbb{D})\rightarrow
\mathcal{L}^{2,1-\frac{2}{p}}(\mathbb{D})$ is bounded. Consider the
test function
$$f(z)=\frac{(1-|b|^{2})^{1-\frac{1}{p}}}{\overline{b}(1-\overline{b}z)}, \ \frac{1}{2}\leq |b|<1, \ z\in \mathbb{D}.$$
Applying the well-known inequality,
$$\int_{0}^{2\pi}\frac{1}{|1-ze^{i\theta}|^{1+t}}d\theta\lesssim \frac{1}{(1-|z|^{2})^{t}},\ \ \  z\in \mathbb{D}, t>0,$$
we have $f\in H^{p}(\mathbb{D})$ with $\|f\|_{H^{p}}\lesssim 1$. So,
\begin{equation}\nonumber
\begin{split}
\|I_{g}\|&\gtrsim \|I_{g}f\|_{\mathcal{L}^{2,
1-\frac{2}{p}}}\\&\approx \sup_{a\in
\mathbb{D}}\Big((1-|a|^{2})^{\frac{2}{p}}\int_{\mathbb{D}}|f'(z)|^{2}|g(z)|^{2}(1-|\sigma_{a}(z)|^{2})dA(z)
\Big)^{1/2}\\&\geq\Big(\int_{\mathbb{D}}
|\sigma_{b}'(z)|^{2}|g(z)|^{2}(1-|\sigma_{b}(z)|^{2})dA(z)\Big)^{1/2}
\\&=\Big(\int_{\mathbb{D}}
|g(\sigma_{b}(w))|^{2}(1-|w|^{2})dA(w)\Big)^{1/2}
\\&\gtrsim |g(b)|,
\end{split}
\end{equation}
where we used Lemma 4.12 of \cite{zhu} again in the last inequality.
Since $\frac{1}{2}\leq |b|<1 $, we have
$$\|g\|_{\infty}=\sup_{\frac{1}{2}\leq |z|<1}|g(z)|\lesssim \|I_{g}\|.$$

On the other hand, suppose $g\in H^{\infty}(\mathbb{D})$, then
applying Littlewood-Paley identity, we get
\begin{equation}\nonumber
\begin{split}
\|I_{g}f\|_{\mathcal{L}^{2, 1-\frac{2}{p}}}&\approx \sup_{a\in
\mathbb{D}}\Big((1-|a|^{2})^{\frac{2}{p}}\int_{\mathbb{D}}|f'(z)|^{2}|g(z)|^{2}(1-|\sigma_{a}(z)|^{2})dA(z)\Big)^{1/2}
\\&\lesssim\sup_{a\in
\mathbb{D}}\Big((1-|a|^{2})^{\frac{2}{p}}\int_{\mathbb{D}}|f'(z)|^{2}|g(z)|^{2}\log\frac{1}{|\sigma_{a}(z)|}
dA(z)\Big)^{1/2}
\\&\lesssim \|g\|_{\infty}\sup_{a\in
\mathbb{D}}(1-|a|^{2})^{\frac{1}{p}}\|f\circ\sigma_{a}-f(a)\|_{H^{2}}
\\&\lesssim \|g\|_{\infty}\sup_{a\in
\mathbb{D}}(1-|a|^{2})^{\frac{1}{p}}\|f\circ\sigma_{a}\|_{H^{2}}
\\&\lesssim\|g\|_{\infty} \sup_{a\in
\mathbb{D}}(1-|a|^{2})^{\frac{1}{p}}\|f\circ\sigma_{a}\|_{H^{p}}
\\&\lesssim \|g\|_{\infty}\|f\|_{H^{p}},
\end{split}
\end{equation}
where we used the fact that
$$\|f\circ\sigma_{a}\|_{H^{p}}\leq \Big(\frac{1+|a|}{1-|a|}\Big)^{1/p}\|f\|_{H^{p}}$$
(\cite[Theorem 3.6]{CM}). Therefore
$$\|I_{g}\|\lesssim \|g\|_{\infty}.$$
The proof is completed.
\end{proof}

In \cite[Theorem 9]{WZ}, Wu proved that, for $g\in H(\mathbb{D})$,
$T_{g}: H^{p}(\mathbb{D})\rightarrow
\mathcal{L}^{2,1-\frac{2}{p}}(\mathbb{D}) (2<p\leq \infty)$ is
bounded if and if $g\in BMOA$. We next estimate the norm of $T_{g}$.

\begin{theorem}\label{th3}
Let $2<p\leq \infty$ and $g\in H(\mathbb{D})$. If $T_{g}:\
H^{p}(\mathbb{D})\rightarrow
\mathcal{L}^{2,1-\frac{2}{p}}(\mathbb{D})$ is bounded, then
$\|T_{g}\|\approx \|g\|_{BMOA}$.
\end{theorem}
\begin{proof}
The proof is similar to that of Theorem $\ref{th2}$, for the
completeness of the paper, we give the sketch of the proof below.
For any $I\subset
\partial \mathbb{D}$, let $\zeta$ be the center of $I$ and
$b=(1-|I|)\zeta$. Consider the function
$$h_{b}(z)=\frac{1-|b|^{2}}{(1-\overline{b}z)^{1+\frac{1}{p}}},\ \ z\in \mathbb{D}.$$
It is easy to see that $h_{b}\in H^{p}(\mathbb{D})$ with
$\|h_{b}\|_{H^{p}}\lesssim 1$. So,
\begin{equation}\nonumber
\begin{split}
\frac{1}{|I|}\int_{S(I)}|g'(z)|^{2}(1-|z|^{2})dA(z) &\lesssim
\frac{1}{|I|^{1-\frac{2}{p}}}\int_{S(I)}|h_{b}(z)|^{2}|g'(z)|^{2}(1-|z|^{2})dA(z)\\&
=\frac{1}{|I|^{1-\frac{2}{p}}}\int_{S(I)}|(T_{g}h_{b})'(z)|^{2}(1-|z|^{2})dA(z)
\\&\leq \|T_{g}h_{b}\|_{\mathcal{L}^{2,1-\frac{2}{p}}}^{2}
\\&\leq \|T_{g}\|^{2}\|h_{b}\|_{H^{p}}^{2}
\\&\lesssim \|T_{g}\|^{2}.
\end{split}
\end{equation}
It follows that $g\in BMOA$ with $\|g\|_{BMOA}\lesssim \|T_{g}\|$.

On the other hand, if $p=\infty$, then, for  $f\in
H^{\infty}(\mathbb{D})$,
\begin{equation}\nonumber
\begin{split}
\frac{1}{|I|}\int_{S(I)}|(T_{g}f)'(z)|^{2}(1-|z|^{2})dA(z) \lesssim
\|f\|_{\infty}^{2}\|g\|_{BMOA}^{2},
\end{split}
\end{equation}
which implies $$\|T_{g}\|\lesssim \|g\|_{BMOA}.$$

If $2<p<\infty$, for $f\in H^{p}(\mathbb{D})$, write
\begin{equation}\nonumber
\begin{split}
\frac{1}{|I|^{1-\frac{2}{p}}}&\int_{S(I)}|(T_{g}f)'(z)|^{2}(1-|z|^{2})dA(z)
\\&\leq\frac{2}{|I|^{1-\frac{2}{p}}}\int_{S(I)}|f(b)|^{2}|g'(z)|^{2}(1-|z|^{2})dA(z)
\\&\qquad +\frac{2}{|I|^{1-\frac{2}{p}}}\int_{S(I)}|f(z)-f(b)|^{2}|g'(z)|^{2}(1-|z|^{2})dA(z)
\\&\triangleq E_{1}+E_{2}.
\end{split}
\end{equation}

Note that (Theorem 9.1 of \cite{zhu}),
$$|f(b)|\leq \frac{\|f\|_{H^{p}}}{(1-|b|^{2})^{\frac{1}{p}}}\lesssim \frac{\|f\|_{H^{p}}}{|I|^{\frac{1}{p}}},\ \ b\in \mathbb{D}.$$
We have
$$E_{1}\lesssim \|g\|_{BMOA}^{2}\|f\|_{H^{p}}^{2}.$$

Since $1-|\sigma_{b}(z)|^{2}\gtrsim \frac{1-|z|^{2}}{|I|}$ for all
$z\in S(I)$, then
\begin{equation}\nonumber
\begin{split}
E_{2}&\lesssim
|I|^{\frac{2}{p}}\int_{S(I)}|f(z)-f(b)|^{2}|g'(z)|^{2}(1-|\sigma_{b}(z)|^{2})dA(z)
\\&\lesssim (1-|b|^{2})^{\frac{2}{p} }\int_{\mathbb{D}}|f\circ\sigma_{b}(w)-f(b)|^{2}|(g\circ
\sigma_{b})'(z)|^{2}(1-|w|^{2})dA(w).
\end{split}
\end{equation}
If $g\in BMOA$, then $|(g\circ \sigma_{b})'(z)|^{2}(1-|w|^{2})dA(w)$
is a Carleson measure. For $p>2$, $f\in H^{p}(\mathbb{D})$ implies
that $f\circ\sigma_{b}\in H^{2}(\mathbb{D})$. Using Theorem A yields
\begin{equation}\nonumber
\begin{split}
E_{2}&\lesssim(1-|b|^{2})^{\frac{2}{p}}
\|g\circ\sigma_{b}\|_{BMOA}^{2}\|f\circ\sigma_{b}-f(b)\|_{H^{2}}^{2}
\\&\lesssim (1-|b|^{2})^{\frac{2}{p}}
\|g\|_{BMOA}^{2}\|f\circ\sigma_{b}\|_{H^{2}}^{2}
\\&\leq (1-|b|^{2})^{\frac{2}{p}}
\|g\|_{BMOA}^{2}\|f\circ\sigma_{b}\|_{H^{p}}^{2}
\\&\lesssim\|g\|_{BMOA}^{2}\|f\|_{H^{p}}^{2},
\end{split}
\end{equation}
where we used the fact that
$$\|f\circ\sigma_{b}\|_{H^{p}}\leq \Big(\frac{1+|b|}{1-|b|}\Big)^{1/p}\|f\|_{H^{p}}$$
(\cite[Theorem 3.6]{CM}).
 Hence
$$\|T_{g}f\|_{\mathcal{L}^{2,1-\frac{2}{p}}}\lesssim \|g\|_{BMOA}\|f\|_{H^{p}}.$$
Therefore
$$\|T_{g}\|\lesssim \|g\|_{BMOA}.$$
 The theorem is proved.
\end{proof}

\begin{remark}\label{rem3}
If $p=2$, then
$\mathcal{L}^{2,1-\frac{2}{p}}(\mathbb{D})=H^{2}(\mathbb{D})$. From
the proof of Theorem $\ref{th31}$, we know that Theorem $\ref{th31}$
holds for $p=2$.  Theorem $\ref{th3}$ for $p=2$ is also true (see,
for example \cite{AC} or \cite{PC}).
\end{remark}

\begin{corollary}
Let $2<p\leq \infty$ and $g\in H(\mathbb{D})$. Then the
multiplication operator $M_{g}:\ H^{p}(\mathbb{D})\rightarrow
\mathcal{L}^{2,1-\frac{2}{p}}(\mathbb{D})$ is bounded if and only if
$g\in H^{\infty}(\mathbb{D})$.
\end{corollary}
\begin{proof}
If $M_{g}:\ H^{p}(\mathbb{D})\rightarrow
\mathcal{L}^{2,1-\frac{2}{p}}(\mathbb{D})$ is bounded, consider the
function
$$f(z)=\frac{(1-|b|^{2})^{1-\frac{1}{p}}}{(1-\overline{b}z)}.$$
It is easy to see that $f\in H^{p}(\mathbb{D})$ and
$\|f\|_{H^{p}}\lesssim 1$. Applying Lemma $\ref{le2}$, we get
\begin{equation}\nonumber
\begin{split}
\Big|\frac{(1-|b|^{2})^{1-\frac{1}{p}}}{(1-\overline{b}z)}g(z)\Big|&\lesssim
\|M_{g}f\|_{\mathcal{L}^{2,1-\frac{2}{p}}}\frac{1}{(1-|z|^{2})^{\frac{1}{p}}}
\\&\lesssim \|M_{g}\|\frac{1}{(1-|z|^{2})^{\frac{1}{p}}}.
\end{split}
\end{equation}
Taking $z=b$ in the inequality above yields $|g(b)|\lesssim
\|M_{g}\|$. Since $b$ is arbitrary in $\mathbb{D}$, we have $g\in
H^{\infty}(\mathbb{D})$.

If $g\in H^{\infty}(\mathbb{D})$, from \cite[Theorem 9]{WZ} and
Theorem 3, we know that $I_{g}$ and $T_{g}$ are both bounded from
$H^{p}(\mathbb{D})$ to $\mathcal{L}^{2,1-\frac{2}{p}}(\mathbb{D})$.
So $M_{g}:\ H^{p}(\mathbb{D})\rightarrow
\mathcal{L}^{2,1-\frac{2}{p}}(\mathbb{D})$ is bounded.
\end{proof}

\section{Essential norm of $I_{g}$ and $T_{g}$}
Let $X$ be a Banach space and $T$ is a bounded linear operator on
$X$. The essential norm of $T$ (denoted by $\|T\|_{e}$) is defined
as follows,
$$\|T\|_{e}=\inf\{\|T-K\|: K\ \mbox{are compact operators on}\ X\}.$$
Since that $T$ is compact if and only if $\|T\|_{e}=0$, then the
estimation of $\|T\|_{e}$ indicates the condition for $T$ to be
compact. In this section, we estimate the essential norm of $I_{g},
T_{g}$.

Let $X$ and $Y$ be two Banach spaces with $X\subset Y$. If $f\in Y$,
then the distance from functions $f$ to the Banach space $X$ is
defined as
$$\mbox{dist}_{Y}(f, X)=\inf_{g\in X}\|f-g\|_{Y}.$$

 In the following lemma, Laitila, Miihkinen and Nieminen (\cite[Lemma 3]{JSP}) characterized the distance
 from functions $f\in BMOA$ to $VMOA$ space. Here and afterward we denote $g_{r}(z)=g(rz)$ with
$0<r<1$.
 \begin{lemma}\label{le5}
 Let $g\in BMOA$. Then
 $$\dist(g, VMOA)\approx \limsup_{r\rightarrow 1}\|g-g_{r}\|_{BMOA}\approx \limsup_{|a|\rightarrow 1}\|g\circ\sigma_{a}-g(a)\|_{H^{2}}.$$
 \end{lemma}

\begin{theorem}\label{th4}
Suppose $0<\lambda<1$ and $g\in H(\mathbb{D})$. If $I_{g}$ is a
bounded operator on $\mathcal{L}^{2, \lambda}(\mathbb{D})$, then
$$\|I_{g}\|_{e}\approx \|g\|_{\infty}.$$
\end{theorem}
\begin{proof}
For compact operators $K$, it follows from Theorem $\ref{th1}$ that
$$\|I_{g}\|_{e}=\inf_{K} \|I_{g}-K\|\leq \|I_{g}\|\lesssim \|g\|_{\infty}.$$

On the other hand, choose the sequence $\{b_{n}\}\subset \mathbb{D}$
such that $|b_{n}|\rightarrow 1$ as $n\rightarrow \infty$. Consider
the sequence of functions
$f_{n}(z)=(1-|b_{n}|^{2})^{\frac{\lambda-1}{2}}(\sigma_{b_{n}}(z)-b_{n})$.
It follows from Lemma $\ref{le4}$ that $\|f_{n}\|_{\mathcal{L}^{2,
\lambda}}\lesssim 1$. Note that $f_{n}$ can be written as
$$f_{n}(z) =-(1-|b_{n}|^{2})^{\frac{\lambda+1}{2}}\int_{0}^{z}\frac{dw}{(1-\overline{b_{n}}w)^{2}}$$
and $f_{n}$ converges to zero uniformly on compact subsets of
$\mathbb{D}$. Then $\|Kf_{n}\|_{\mathcal{L}^{2, \lambda}
}\rightarrow 0$ as $n\rightarrow \infty$ for any compact operator
$K$ on $\mathcal{L}^{2, \lambda}(\mathbb{D})$. Since
\begin{equation}\nonumber
\begin{split}
\|I_{g}-K\|&\gtrsim\limsup_{n\rightarrow
\infty}\|(I_{g}-K)f_{n}\|_{\mathcal{L}^{2, \lambda}}\\&\geq
\limsup_{n\rightarrow \infty}(\|I_{g}f_{n}\|_{\mathcal{L}^{2,
\lambda}}-\|Kf_{n}\|_{\mathcal{L}^{2, \lambda}})\\&\geq
\limsup_{n\rightarrow \infty}\|I_{g}f_{n}\|_{\mathcal{L}^{2,
\lambda}}
\end{split}
\end{equation}
and
\begin{equation}\nonumber
\begin{split}
\|I_{g}f_{n}\|_{\mathcal{L}^{2, \lambda}}&\approx \sup_{a\in
\mathbb{D}}\Big((1-|a|^{2})^{1-\lambda}\int_{\mathbb{D}}|f_{n}'(z)|^{2}|g(z)|^{2}(1-|\sigma_{a}(z)|^{2})dA(z)
\Big)^{1/2}\\&\geq\Big(\int_{\mathbb{D}}
|\sigma_{b_{n}}'(z)|^{2}|g(z)|^{2}(1-|\sigma_{b_{n}}(z)|^{2})dA(z)\Big)^{1/2}
\\&=\Big(\int_{\mathbb{D}}
|g(\sigma_{b_{n}}(w))|^{2}(1-|w|^{2})dA(w)\Big)^{1/2}
\\&\gtrsim |g(b_{n})|.
\end{split}
\end{equation}
We get
$$\|I_{g}\|_{e}\gtrsim\limsup_{n\rightarrow \infty}|g(b_{n})|.$$ The arbitrary choice
of the sequence $\{b_{n}\}$ implies
$\|I_{g}\|_{e}\gtrsim\|g\|_{\infty}$. The proof of Theorem
$\ref{th4}$ is completed.
\end{proof}

For the proof of the next theorem, we need the following technical
lemma.

\begin{lemma}\label{le6}
Suppose $0<\lambda<1$. If $g\in BMOA$, then $T_{g_{r}}:
\mathcal{L}^{2, \lambda}(\mathbb{D})\rightarrow \mathcal{L}^{2,
\lambda}(\mathbb{D})$ is compact.
\end{lemma}
\begin{proof}
 Let $\{f_{n}\}$ be such that
$\|f_{n}\|_{\mathcal{L}^{2, \lambda}}\leq 1$ and $f_{n}\rightarrow
0$ uniformly on compact subsets of $\mathbb{D}$ as $n\rightarrow
\infty$. We need only to show that
$$\lim_{n\rightarrow \infty}\|T_{g_{r}}f_{n}\|_{ \mathcal{L}^{2, \lambda}}=0.$$
Since $\|g_{r}\|_{BMOA}\lesssim \|g\|_{BMOA} $ (see Lemma 1 of
\cite{YL} ),  for $z\in \mathbb{D}$ $$|g_{r}'(z)|\lesssim
\frac{\|g\|_{BMOA}}{1-r^{2}}.$$ We get
\begin{equation}\nonumber
\begin{split}
\|T_{g_{r}}f_{n} \|_{ \mathcal{L}^{2, \lambda}} &\lesssim \sup_{a\in
\mathbb{D}}\Big(
(1-|a|^{2})^{1-\lambda}\int_{\mathbb{D}}|f_{n}(z)|^{2}|g_{r}'(z)|^{2}(1-|\sigma_{a}(z)|^{2})dA(z)\Big)^{1/2}
\\&\lesssim \frac{\|g\|_{BMOA}}{1-r^{2}}\sup_{a\in
\mathbb{D}}\Big(
\int_{\mathbb{D}}|f_{n}(z)|^{2}\frac{(1-|z|^{2})(1-|a|^{2})^{2-\lambda}}{|1-\overline{a}z|^{2}}dA(z)\Big)^{1/2}
\\&\lesssim\frac{\|g\|_{BMOA}}{1-r^{2}}\Big(
\int_{\mathbb{D}}|f_{n}(z)|^{2}(1-|z|^{2})^{1-\lambda}dA(z)\Big)^{1/2}.
\end{split}
\end{equation}
Note that $\|f_{n}\|_{\mathcal{L}^{2, \lambda}}\leq 1$ and
$|f_{n}(z)|^{2}(1-|z|^{2})^{1-\lambda}\lesssim 1$ by Lemma
$\ref{le2}$. The desired result follows from the Dominated
Convergence Theorem.
\end{proof}

In the following theorem, as a main result of this section, we give
the essential norm of $T_{g}$ on
$\mathcal{L}^{2,\lambda}(\mathbb{D})$.

\begin{theorem}\label{th5}
Suppose $0<\lambda<1$ and $g\in BMOA$. Then $T_{g}: \mathcal{L}^{2,
\lambda}(\mathbb{D})\rightarrow \mathcal{L}^{2,
\lambda}(\mathbb{D})$ satisfies
$$\|T_{g}\|_{e}\approx \dist(g, VMOA)\approx \limsup_{|a|\rightarrow 1}\|g\circ\sigma_{a}-g(a)\|_{H^{2}}.$$
\end{theorem}
\begin{proof}
Let $\{I_{n}\}$ be the subarc sequence of $\partial \mathbb{D}$ such
that $|I_{n}|\rightarrow 0$ as $n\rightarrow \infty$. Denote
$b_{n}=(1-|I_{n}|)\zeta_{n}\in \mathbb{D}$, where $\zeta_{n}\in
\partial \mathbb{D}$ is the center of arc $I_{n},\ n=1,2, \cdots$. We know that
$$(1-|b_{n}|^{2})\approx|1-\overline{b_{n}}z|\approx|I_{n}|,\ \ z\in S(I_{n}).$$
Choose the function
$F_{n}(z)=(1-|b_{n}|^{2})(1-\overline{b_{n}}z)^{\frac{\lambda-3}{2}},
z\in \mathbb{D}$. Then $F_{n}\rightarrow 0$ uniformly on the compact
subsets of $\mathbb{D}$ as $n\rightarrow \infty$ and
$\|F_{n}\|_{\mathcal{L}^{2, \lambda}}\lesssim 1$  by Lemma
$\ref{le4}$. Thus,
$$\lim_{n\rightarrow \infty} \|K F_{n}\|_{\mathcal{L}^{2,
\lambda}}\rightarrow 0$$ for any compact operator $K$ on $
\mathcal{L}^{2, \lambda}(\mathbb{D})$. Therefore
\begin{equation}\nonumber
\begin{split}
\|T_{g}-K\|&\gtrsim \limsup_{n\rightarrow
\infty}(\|T_{g}F_{n}\|_{\mathcal{L}^{2, \lambda}}-\|
KF_{n}\|_{\mathcal{L}^{2, \lambda}})
\\&=\limsup_{n\rightarrow
\infty}\|T_{g}F_{n}\|_{\mathcal{L}^{2, \lambda}}
\\&\gtrsim \limsup_{n\rightarrow
\infty} \Big(
\frac{1}{|I_{n}|^{\lambda}}\int_{S(I_{n})}|F_{n}(z)|^{2}|g'(z)|^{2}(1-|z|^{2})dA(z)
\Big)^{1/2}
\\&\approx \limsup_{n\rightarrow
\infty} \Big(
\frac{1}{|I_{n}|}\int_{S(I_{n})}|g'(z)|^{2}(1-|z|^{2})dA(z)
\Big)^{1/2}.
\end{split}
\end{equation}
Since the sequence $\{I_{n}\}$ is arbitrary, we obtain
$$\|T_{g}\|_{e}\gtrsim \limsup_{|I|\rightarrow 0} \Big( \frac{1}{|I|}\int_{S(I)}|g'(z)|^{2}(1-|z|^{2})dA(z)
\Big)^{1/2}.$$  It follows from the proof of Lemma 3.4 in \cite{SZ}
that, for $g\in BMOA$,
$$\limsup_{|a|\rightarrow 1}\|g\circ\sigma_{a}-g(a)\|_{H^{2}}\approx  \limsup_{|I|\rightarrow 0} \Big( \frac{1}{|I|}\int_{S(I)}|g'(z)|^{2}(1-|z|^{2})dA(z)
\Big)^{1/2}.$$ Hence
$$\|T_{g}\|_{e}\gtrsim \limsup_{|a|\rightarrow 1}\|g\circ\sigma_{a}-g(a)\|_{H^{2}}.$$

On the other hand,  $T_{g_{r}}:\mathcal{L}^{2,
\lambda}(\mathbb{D})\rightarrow \mathcal{L}^{2, \lambda}(\mathbb{D})
$ is a compact operator. Combining this with Theorem $\ref{th2}$
implies
$$\|T_{g}\|_{e}\leq \|T_{g}-T_{g_{r}}\|=\|T_{g-g_{r}}\|\approx\|g-g_{r}\|_{BMOA},$$
where we used the linearity of $T_{g}$ respect to $g$. Hence
$$\|T_{g}\|_{e}\lesssim \limsup_{r\rightarrow 1}\|g-g_{r}\|_{BMOA}\approx\limsup_{|a|\rightarrow 1}\|g\circ\sigma_{a}-g(a)\|_{H^{2}}$$
by Lemma $\ref{le5}$. The theorem is proved.
\end{proof}

\begin{remark}
For $\lambda=1$, $\mathcal{L}^{2,\lambda}(\mathbb{D})=BMOA$. Theorem
$\ref{th4}$ holds for $\lambda=1$, see Theorem 2.1 of \cite{LLX}.
But Theorem $\ref{th5}$ is not true for $\lambda=1$ (Theorem 2 of
\cite{JSP}).
\end{remark}

\begin{corollary}
Let $0<\lambda<1$ and $g\in H(\mathbb{D})$. Then
\newline (1)\ $I_{g}$
is compact on $\mathcal{L}^{2, \lambda}(\mathbb{D})$ if and only if
$g=0$; \newline (2)\ $T_{g}$ is compact on $\mathcal{L}^{2,
\lambda}(\mathbb{D})$ if and only if $g\in VMOA$.
\end{corollary}

Now we give the essential norm of $I_{g}:
H^{p}(\mathbb{D})\rightarrow
\mathcal{L}^{2,1-\frac{2}{p}}(\mathbb{D})$ ($2<p\leq \infty$) in the
following theorem.
\begin{theorem}\label{th7}
Let $2<p\leq \infty$ and $g\in H(\mathbb{D})$. If $I_{g}:
H^{p}(\mathbb{D})\rightarrow
\mathcal{L}^{2,1-\frac{2}{p}}(\mathbb{D})$ is bounded, then
$$\|I_{g}\|_{e}\approx \|g\|_{\infty}.$$
\end{theorem}
\begin{proof}
The proof is similar to that of Theorem $\ref{th31}$. Choose the
sequence $\{b_{n}\}\subset \mathbb{D}$ with $|b_{n}|\geq 1/2$ such
that $|b_{n}|\rightarrow 1$ as $n\rightarrow \infty$. Consider the
sequence of functions
$$f_{n}(z)=\frac{(1-|b_{n}|^{2})^{1-\frac{1}{p}}}{\overline{b_{n}}(1-\overline{b_{n}}z)}.$$
It is easy to see that $\|f_{n}\|_{H^{p}}\lesssim 1$ and $f_{n}$
converges to zero uniformly on compact subsets of $\mathbb{D}$. Then
$\|Kf_{n}\|_{\mathcal{L}^{2,1-\frac{2}{p}}}\rightarrow 0$ as
$n\rightarrow \infty$ for any compact operator $K:
H^{p}(\mathbb{D})\rightarrow
\mathcal{L}^{2,1-\frac{2}{p}}(\mathbb{D})$. Since
\begin{equation}\nonumber
\begin{split}
\|I_{g}-K\|&\gtrsim\limsup_{n\rightarrow
\infty}\|(I_{g}-K)f_{n}\|_{\mathcal{L}^{2,1-\frac{2}{p}}}\\&\geq
\limsup_{n\rightarrow
\infty}(\|I_{g}f_{n}\|_{\mathcal{L}^{2,1-\frac{2}{p}}}-\|Kf_{n}\|_{\mathcal{L}^{2,1-\frac{2}{p}}})\\&\geq
\limsup_{n\rightarrow
\infty}\|I_{g}f_{n}\|_{\mathcal{L}^{2,1-\frac{2}{p}}}
\end{split}
\end{equation}
and
\begin{equation}\nonumber
\begin{split}
\|I_{g}f_{n}\|_{\mathcal{L}^{2,1-\frac{2}{p}}}&\approx \sup_{a\in
\mathbb{D}}\Big((1-|a|^{2})^{\frac{2}{p}}\int_{\mathbb{D}}|f_{n}'(z)|^{2}|g(z)|^{2}(1-|\sigma_{a}(z)|^{2})dA(z)
\Big)^{1/2}\\&\geq\Big(\int_{\mathbb{D}}
|\sigma_{b_{n}}'(z)|^{2}|g(z)|^{2}(1-|\sigma_{b_{n}}(z)|^{2})dA(z)\Big)^{1/2}
\\&=\Big(\int_{\mathbb{D}}
|g(\sigma_{b_{n}}(w))|^{2}(1-|w|^{2})dA(w)\Big)^{1/2}
\\&\gtrsim |g(b_{n})|.
\end{split}
\end{equation}
We have
$$\|I_{g}\|_{e}\gtrsim \limsup_{n\rightarrow \infty}|g(b_{n})|.$$
The arbitrary choice of the sequence $\{b_{n}\}$ implies
$\|I_{g}\|_{e}\gtrsim \|g\|_{\infty}$.

On the other hand, for compact operators $K$, it follows from
Theorem $\ref{th31}$ that
$$\|I_{g}\|_{e}=\inf_{K}\|I_{g}-K\|\leq \|I_{g}\|\approx\|g\|_{\infty}.$$
The theorem is proved.
\end{proof}

As another main result of this section, we next estimate essential
norm of $T_{g}:\ H^{p}(\mathbb{D})\rightarrow
\mathcal{L}^{2,1-\frac{2}{p}}(\mathbb{D})\ (2<p\leq \infty)$, for
its proof, we need the following lemma.

\begin{lemma}\label{le7}
Let $2<p\leq \infty$ and $g\in BMOA$. Then $T_{g_{r}}:\
H^{p}(\mathbb{D})\rightarrow
\mathcal{L}^{2,1-\frac{2}{p}}(\mathbb{D})$ is compact.
\end{lemma}
\begin{proof}
The proof is similar to that of Lemma 6.  Let $\{f_{n}\}\subset
H^{p}(\mathbb{D})$ be such that $\|f_{n}\|_{H^{p}}\leq 1$ and
$f_{n}\rightarrow 0$ uniformly on compact subsets of $\mathbb{D}$ as
$n\rightarrow \infty$. We have
\begin{equation}\nonumber
\begin{split}
\|T_{g_{r}}f_{n}\|_{\mathcal{L}^{2,1-\frac{2}{p}}}&\lesssim
\sup_{a\in
\mathbb{D}}\Big((1-|a|^{2})^{\frac{2}{p}}\int_{\mathbb{D}}|f_{n}(z)|^{2}|g_{r}'(z)|^{2}(1-|\sigma_{a}|^{2})dA(z)
\Big)^{1/2}
\\&\lesssim\frac{\|g\|_{BMOA}}{1-r^{2}}\sup_{a\in \mathbb{D}}
\Big(\int_{\mathbb{D}}|f_{n}(z)|^{2}\frac{(1-|z|^{2})(1-|a|^{2})^{1+\frac{2}{p}}}{|1-\overline{a}z|^{2}}dA(z)
\Big)^{1/2}
\\&\lesssim \frac{\|g\|_{BMOA}}{1-r^{2}}
\Big(\int_{\mathbb{D}}|f_{n}(z)|^{2}(1-|z|^{2})^{\frac{2}{p}}dA(z)\Big)^{1/2}
\end{split}
\end{equation}
and $|f_{n}(z)|(1-|z|^{2})^{\frac{2}{p}}\lesssim 1$ by Theorem 9.1
of \cite{zhu}. Using the Dominated Convergence Theorem again implies
$\lim\limits_{n\rightarrow
\infty}\|T_{g_{r}}f_{n}\|_{\mathcal{L}^{2,1-\frac{2}{p}}}=0$.
\end{proof}

\begin{theorem}\label{th6}
Suppose $2<p\leq \infty$ and $g\in BMOA$. Then $T_{g}:
H^{p}(\mathbb{D})\rightarrow
\mathcal{L}^{2,1-\frac{2}{p}}(\mathbb{D})$ satisfies
$$\|T_{g}\|_{e}\approx \dist(g,VMOA)\approx \limsup_{|a|\rightarrow 1}\|g\circ \sigma_{a}-g(a)\|_{H^{2}}.$$
\end{theorem}
\begin{proof}
 As in the proof of Theorem $\ref{th5}$. Let $\{I_{n}\}$ be the subarc sequence of $\partial \mathbb{D}$
 such that $|I_{n}|\rightarrow 0$ as $n\rightarrow \infty$,
 $b_{n}=(1-|I_{n}|\zeta_{n})\in \mathbb{D}$ and  $\zeta_{n}\in \partial
 \mathbb{D}$ be the center of arc $I_{n}$.
 Consider the function
 $$h_{n}(z)=\frac{1-|b_{n}|^{2}}{(1-\overline{b_{n}}z)^{1+\frac{1}{p}}}.$$
 It is easy to check that  $h_{n}\rightarrow 0$
 uniformly on compact subsets of $\mathbb{D}$ as $n\rightarrow
 \infty$ and $\|h_{n}\|_{H^{p}}\lesssim 1$. For any compact operator $K: H^{p}(\mathbb{D})\rightarrow
\mathcal{L}^{2,1-\frac{2}{p}}(\mathbb{D})$, we have
$$\lim_{n\rightarrow \infty}\|Kh_{n}\|_{ \mathcal{L}^{2,1-\frac{2}{p}}} \rightarrow 0.$$
So,
\begin{equation}\nonumber
\begin{split}
\|T_{g}-K\|&\gtrsim \limsup_{n\rightarrow \infty}(\|T_{g}h_{n}\|_{
\mathcal{L}^{2,1-\frac{2}{p}}}-\|Kh_{n}\|_{
\mathcal{L}^{2,1-\frac{2}{p}}})
\\&=\limsup_{n\rightarrow \infty}\|T_{g}h_{n}\|_{
\mathcal{L}^{2,1-\frac{2}{p}}}
\\&\gtrsim \limsup_{n\rightarrow \infty}\Big(\frac{1}{|I_{n}|^{1-\frac{2}{p}}}\int_{S(I_{n})}|(T_{g}h_{n})'(z)|^{2}(1-|z|^{2})dA(z)\Big)^{1/2}
\\&\approx\limsup_{n\rightarrow \infty}\Big(\frac{1}{|I_{n}|^{1-\frac{2}{p}}}\int_{S(I_{n})}|h_{n}(z)|^{2}|g'(z)|^{2}(1-|z|^{2})dA(z)\Big)^{1/2}
\\&\gtrsim\limsup_{n\rightarrow
\infty}\Big(\frac{1}{|I_{n}|}\int_{S(I_{n})}|g'(z)|^{2}(1-|z|^{2})dA(z)\Big)^{1/2}.
\end{split}
\end{equation}
The arbitrary choice of $\{I_{n}\}$ yields,
\begin{equation}\nonumber
\begin{split}
 \|T_{g}\|_{e}&\gtrsim \limsup_{|I|\rightarrow
0}\Big(\frac{1}{|I|}\int_{S(I)}|g'(z)|^{2}(1-|z|^{2})dA(z)\Big)^{1/2}
\\&\approx \limsup_{|a|\rightarrow 1}\|g\circ
\sigma_{a}-g(a)\|_{H^{2}}.
\end{split}
\end{equation}

On the other hand, it follows from Lemma $\ref{le7}$ that
$T_{g_{r}}: H^{p}(\mathbb{D})\rightarrow
\mathcal{L}^{2,1-\frac{2}{p}}(\mathbb{D})$ is a compact operator.
Applying Theorem $\ref{th3}$ gives
$$\|T_{g}\|_{e}\leq \|T_{g}-T_{g_{r}}\|=\|T_{g-g_{r}}\|\approx \|g-g_{r}\|_{BMOA}.$$
So
$$\|T_{g}\|_{e}\lesssim \limsup_{r\rightarrow 1}\|g-g_{r}\|_{BMOA}\approx \limsup_{|a|\rightarrow 1}\|g\circ \sigma_{a}-g(a)\|_{H^{2}}$$
by Lemma $\ref{le5}$. The proof of Theorem $\ref{th6}$ is finished.
\end{proof}

\begin{remark}
The proof of Theorem $\ref{th7}$ shows that Theorem $\ref{th7}$ is
true for $p=2$. Theorem $\ref{th6}$ holds also for $p=2$ ( Theorem 1
of \cite{JSP}).
\end{remark}

\begin{corollary}
(\cite[Theorem 9]{WZ})  Let $2\leq p\leq \infty$ and $g\in
H(\mathbb{D})$. Then
\newline (1)\ $I_{g}:\ H^{p}(\mathbb{D})\rightarrow
\mathcal{L}^{2,1-\frac{2}{p}}(\mathbb{D})$ is compact if and only if
$g=0$,
\newline(2)\ $T_{g}: H^{p}(\mathbb{D})\rightarrow
\mathcal{L}^{2,1-\frac{2}{p}}(\mathbb{D})$ is compact if and only if
$g\in VMOA$.
\end{corollary}

\end{document}